\documentclass[12pt]{amsart}

\usepackage{amssymb,amsthm,amsmath}
\usepackage{hyperref}
\usepackage{geometry}
\newtheorem{theorem}{Theorem}[section]
\newtheorem{lemma}[theorem]{Lemma}
\newtheorem{example}[theorem]{Example}
\newtheorem{proposition}[theorem]{Proposition}
\newtheorem{remark}[theorem]{Remark}

\newtheorem{definition}[theorem]{Definition}
\numberwithin{equation}{section}
\newenvironment{proof31}[1][\noindent \textbf{Proof of Proposition  \ref{normal form automorphism}: }]{#1}{ \hfill $\square$ \vspace{2mm}}

\begin{document}

\title[Global Gevrey hypoellipticity for systems of vector fields]{Global Gevrey hypoellipticity on the torus for a class of systems of complex vector fields}
	
	\author[A. \'{A}rias]{Alexandre \'{A}rias Junior}
	\address{Programa de P\'{o}s-Gradua\c c\~ao em Matem\'{a}tica \\ Universidade Federal do Paran\'{a} \\ Caixa Postal 19081 \\ 81531-990, Curitiba, PR, Brazil}
	\email{{alexandreaaj@hotmail.com}}
		
	\author[A. Kirilov]{\\ Alexandre Kirilov}
	\address{Departamento de Matem\'{a}tica \\ Universidade Federal do Paran\'{a} \\ Caixa Postal 19081 \\ 81531-990, Curitiba, PR, Brazil}
	\email{akirilov@ufpr.br}
	
	\author[C. de Medeira]{Cleber de Medeira}
	\address{Departamento de Matem\'{a}tica \\ Universidade Federal do Paran\'{a} \\ Caixa Postal 19081 \\ 81531-990, Curitiba, PR, Brazil}
	\email{clebermedeira@ufpr.br}
	
	
	\date{}
	
    \begin{abstract}
	    Let $L_j = \partial_{t_j} + (a_j+ib_j)(t_j) \partial_x, \, j = 1, \dots, n,$ be a system of vector fields defined on the torus $\mathbb{T}_t^{n}\times\mathbb{T}_x^1$, where the coefficients $a_j$ and $b_j$ are real-valued functions belonging to the Gevrey class $G^s(\mathbb{T}^1)$, with $s>1$.  In this paper we were able to characterize the global $s-$hypoellipticity of this  system in terms of Diophantine approximations and the Nirenberg-Treves  condition (P).
    \end{abstract}	
	
	
	\maketitle

\section{Introduction}

The pourpose of this paper is to study the global $s-$hypoellipticity, for $s>1$, of the following system of vector fields
\begin{equation}\label{system}
L_j = \dfrac{\partial}{\partial t_j} + (a_j+ib_j)(t_j)\dfrac{\partial}{\partial x},\;j=1,\ldots,n,
\end{equation}
defined on the  torus $\mathbb{T}^{n+1} \simeq \mathbb{R}^{n+1}/2\pi\mathbb{Z}^{n+1}$, where $a_j$ and $b_j$ are real-valued functions in the Gevrey class $G^s(\mathbb{T}^1_{t_j})$ and $(t, x)=(t_1, \dots, t_n,x) $ denote the coordinates of $\mathbb{T}^{n+1}$.

The system \eqref{system} is said to be globally $s-$hypoelliptic  if the conditions
$u \in D'_{s}(\mathbb{T}^{n+1})$ and  $L_j u \in G^s(\mathbb{T}^{n+1}),$ for $j = 1, \dots, n$, imply that $u \in G^s(\mathbb{T}^{n+1}).$

While the study of the smooth global hypoellipticity of this system is covered by A. Bergamasco, P. Cordaro and P. Malagutti in \cite{BCM93}, the study of the global analytic hypoellipticity of \eqref{system}  is covered by  A. Bergamasco in \cite{Berg99}, however,  by using different techniques from those presented here. In particular, the existence of Gevrey cut-off functions, when $s > 1$, allows us to take a different approach from that considered in \cite{Berg99} with respect to the construction of singular solutions.

When the system \eqref{system} has only a single vector field, the global $s-$hypoel\-lipticity is described by T. Gramchev, P. Popivanov and M. Yoshino in \cite{GPY93}, and also can be obtained as  a particular case of the results of A. Bergamasco, P. Dattori and R. Gonzalez in \cite{BDG18}.


Our main result shows that the global $s-$hypoellipticity of \eqref{system} can be characterized by the Nirenberg-Treves condition (P) and by a Diophantine condition (see Theorem \ref{main theorem}). We also present examples of globally analytic hypoelliptic systems that are not globally $s-$hypoelliptic, for any $s>1,$ and examples of globally $s-$hypoelliptic systems that are not globally hypoelliptic. Finally, we make some connections with related works, such as \cite{BCM93, BdMZ12, Green72} and \cite{Hou79}, among others.

For more results on global properties of systems of vector fields on the torus, we refer the reader to the works  \cite{AlbZan03,AlbZan04,BKNZ15,BCP96,BKNZ12,PtZa08} and the references therein.

\section{Preliminaries and statement of the main result} \label{section-prelim}

Given $s\geqslant 1$ and $h>0$, we say that a smooth, complex-valued periodic function $f$ is in $G^{s,h} (\mathbb{T}^{n+1})$, if there is $C>0$ such that
$$
|\partial^{\alpha} f(t,x)| \leqslant  C h^{|\alpha|} \alpha!^s,
$$
for every multi-index $\alpha \in \mathbb{N}_0^{n+1}$ and $(t,x) \in \mathbb{T}^n_t\times\mathbb{T}^1_x$.

Observe that $G^{s, h}(\mathbb{T}^{n+1})$ is a Banach space with respect to the norm
$$
\| f \|_{s, h} \,= \sup_{\alpha \in \mathbb{N}_0^{n+1}} \sup_{(t,x) \in \mathbb{T}^{n+1}}  |\partial^{\alpha} f(t,x)|h^{-|\alpha|} \alpha!^{-s}
$$
and, for $h<h'$, the inclusion map from $G^{s, h} (\mathbb{T}^{n+1})$  to $G^{s, h'} (\mathbb{T}^{n+1})$ is continuous and compact. The space of periodic Gevrey functions of order $s$ is defined by
$$
G^s (\mathbb{T}^{n+1}) =\displaystyle \underset{h\rightarrow +\infty}{\mbox{ind} \lim} \;G^{s, h} (\mathbb{T}^{n+1}).
$$

The dual space of $G^s (\mathbb{T}^{n+1})$ is denoted by $D'_{s} (\mathbb{T}^{n+1})$ and its elements are called ultradistributions.

In this work we use the well known characterization of the elements of $G^s (\mathbb{T}^{n+1})$  by its Fourier coefficients, namely a smooth periodic function $f$ is in $G^{s}(\mathbb{T}^{n+1})$ if there exist positive constants $C$, $h$ and $\epsilon$ such that
$$
|\partial_{t'}^{\alpha}\hat{f}(t',\eta,\xi)|\leqslant  C h^{|\alpha|}(\alpha !)^{s} e^{-\epsilon |(\eta,\xi)|^{1/s}},
$$
for all $(t',\eta,\xi)\in \mathbb{T}^{\ell}\times \mathbb{Z}^{n-\ell}\times\mathbb{Z}$  and $\alpha \in \mathbb{N}_0^{\ell}$.

\begin{definition}[see \cite{GPY93, AlbPop06, Him01}]\label{E^sL vector}
	We say that $\alpha \in \mathbb{R}^m\setminus\mathbb{Q}^m$ is an exponential Liouville vector of order $s\geqslant  1$, if there exists $\epsilon>0$ such that the inequality 
	$$
	|q \alpha  - p| <  e^{-\epsilon |q|^{{1}/{s}}},
	$$
	has infinitely many solutions $(p,q)\in\mathbb{Z}^m\times \mathbb{Z}^1$.
\end{definition}

Therefore, $\alpha  \in \mathbb{R}^m\setminus\mathbb{Q}^m$ is not an exponential Liouville vector of order $s\geqslant  1$, if for every $\epsilon >0$ there exists $C_\epsilon>0$ such that
$$
|q\alpha - p| \geqslant   C_\epsilon e^{-\epsilon |q|^{{1}/{s}}},
$$
for all $(p,q)\in\mathbb{Z}^m\times \mathbb{Z}^1$.

Now, for each $j=1,\ldots,n$ we denote
\begin{equation}\label{aj0-bj0}
a_{j0}  = \frac{1}{2\pi} \int_{0}^{2\pi} a_j(s)ds \quad \text{and} \quad b_{j0}  = \frac{1}{2\pi} \int_{0}^{2\pi} b_j(s)ds.
\end{equation}
and we set $c_{j0}=a_{j0}+ib_{j0}$.

The main result of this work is the following theorem:

\begin{theorem}\label{main theorem}
	The system
	\begin{equation}\label{system-main-thm}
	L_j = \dfrac{\partial}{\partial t_j} + (a_j+ib_j)(t_j)\dfrac{\partial}{\partial x},\;j=1,\ldots,n,
	\end{equation}
	defined on $\mathbb{T}^{n+1}$, where  $a_j$ and $b_j$ are real-valued functions in $G^s(\mathbb{T}^1_{t_j})$, $s> 1$, is globally $s-$hypoelliptic if and only if at least one of the following conditions occurs:	
	\begin{enumerate}
		\item[(I)] There is $j\in \{1, \dots, n\}$ such that the function $t_j \in \mathbb{T}^1 \mapsto b_j(t_j)$ does not change sign and is not identically zero;
		\item[(II)] If the set
		$$
		J = \left\{ j \in \{1, \dots, n\}: b_j(t_j) \equiv 0 \right\} = \{j_1< \dots< j_\ell \} \neq \emptyset,
		$$
		then $a_{J0} = (a_{j_1 0}, \dots, a_{j_\ell 0})$ is not an exponential Liouville vector of order $s$ and is not in $\mathbb{Q}^\ell.$
	\end{enumerate}
\end{theorem}

The proof of this theorem is divided into the following sections. First, in Section \ref{section-normal-form}, we show that it is sufficient to consider the case where the real part of this system has constant coefficients. Next, in Section \ref{section-sufficiency}, we prove the sufficiency by estimating the Fourier coefficients of the solution of the system. Finally, in Section \ref{section-necessity}, we prove the necessity by constructing singular solutions for the system \eqref{system-main-thm} when the conditions (I) and (II) fail.

\section{Normal form} \label{section-normal-form}

The goal of this section is to show that in the study of  the global $s-$hypoel\-lipticity of the system \eqref{system-main-thm}, it is enough to consider the case where the functions $a_j$ are constants, more precisely, we will show that the system \eqref{system-main-thm} is globally $s-$hypoelliptic if and only if the system
\begin{equation*}
\widetilde{L}_j = \dfrac{\partial}{\partial {t_j} } + (a_{j0} + ib_j(t_j))\dfrac{\partial}{\partial x }, \ j=1,\ldots,n,
\end{equation*}
is globally $s-$hypo\-elliptic, with $a_{j0}$ and $b_{j0}$ defined in \eqref{aj0-bj0}.

This follows immediately from the next proposition.

\begin{proposition}\label{normal form automorphism}
	The operator  $T :  D'_{s}(\mathbb{T}^{n+1}) \to D'_{s}(\mathbb{T}^{n+1})$  given by
	\begin{equation*}
	u(t,x)=\sum_{\xi\in\mathbb{Z}}\widehat{u}(t,\xi)  e^{i\xi x} \longmapsto Tu(t,x)=\sum_{\xi\in\mathbb{Z}}\widehat{u}(t,\xi)e^{i\xi A(t)}  e^{i\xi x},
	\end{equation*}
	where
	\begin{eqnarray}\label{A definition}
	A(t) = \sum_{j=1}^{n} \int_{0}^{t_j} a_j(s)ds -a_{j0}t_j,
	\end{eqnarray}
	defines an automorphism in $D'_{s}(\mathbb{T}^{n+1})$ (and in $G^{s}(\mathbb{T}^{n+1})$).
	
	Furthermore, the following conjugation holds
	$$
	\widetilde{L}_j = T L_j T^{-1},
	$$
	for $j=1,\ldots,n.$
\end{proposition}

Before proceeding with the proof of this result, let us present three technical lemmas.

Let $m\in\mathbb{N}$. Consider the set
\begin{equation}\label{delta definition}
\Delta(m)=\{(k_1,\ldots,k_m)\in\mathbb{N}^m_0: k_1+2k_2+\cdots+mk_m=m\}.
\end{equation}

\begin{lemma}\label{lema_estimativa_produtorio_vzs_fatorial}
	If $(k_1,\ldots,k_m) \in \Delta(m)$ and $k\doteq k_1+\cdots +k_m$ then, for any $s>1$, we have
	$$
	(k!)^{s} \prod_{\ell=1}^{m} \ell !^{(s-1)k_\ell} \leqslant  k! \;m!^{s-1}.
	$$
\end{lemma}

\begin{lemma}\label{lema_sominha_maluca}
	For each $m\in \mathbb{N}$ and $R \in \mathbb{R}$ we have
	$$
	\sum _{\Delta(m)} \dfrac{k!}{k_1!k_2!\ldots k_m!} R^{k} = R(1+R)^{m-1},
	$$
	where $k\doteq k_1+\cdots +k_m$.
\end{lemma}

The proof of Lemma \ref{lema_estimativa_produtorio_vzs_fatorial} can be found in \cite{BDG18} (see Lemma 2.1) and the proof of Lemma \ref{lema_sominha_maluca}  in \cite{KrPar02-book} (see Lemma 1.3.2).

\begin{lemma}\label{important lemma}
	Given $\epsilon>0$, there exists $C_{\epsilon}>0$ such that
	$$
	e^{-\epsilon|\xi|^{1/s}}\left| \partial_{t}^{\alpha} e^{i\xi A(t)}\right|\leqslant   C_\epsilon^{|\alpha|} ({\alpha} !)^s,
	$$
	for all $(t,\xi)\in\mathbb{T}^{n}\times \mathbb{Z}$ and $\alpha\in\mathbb{N}_0^{n},$ where $A(t)$ is defined in \eqref{A definition}.
\end{lemma}

\begin{proof}
	Let $\epsilon>0$  and $\alpha=(\alpha_1,\ldots,\alpha_n)\in\mathbb{N}_0^{n}$.  Since $A(t)=\sum_{j=1}^{n}A_j(t_j)$, where $$A_j(t_j)=\int_{0}^{t_j}a_j(\tau)d\tau-a_{j0}t_j,$$
	then it is enough to prove that: for each $j$ there exists $C_{\epsilon,j}>0$ such that
	$$
	e^{-\widetilde{\epsilon}|\xi|^{1/s}}\left| \partial_{t_j}^{\alpha_j} e^{i\xi A_j(t_j)}\right|\leqslant   C_{\epsilon,j}^{\alpha_j} ({\alpha_j} !)^s,
	$$
	for all $(t_j,\xi)\in\mathbb{T}^{1}\times \mathbb{Z}$ and $\alpha_j\in\mathbb{N}_0,$ where $\widetilde{\epsilon}=\epsilon/n$.
	
	By fixing $j$, since $A_j\in G^{s}(\mathbb{T}^1)$, then there exist $C_j>0$ and $h_j>0$ such that
	$$
	|\partial_{t_j}^{\ell}A_j(t_j)|\leqslant  C_j h_j^{\ell}(\ell!)^s,
	$$
	for all $t_j\in\mathbb{T}^1$ and $\ell\in \mathbb{N}_{0}$.
	
	Also,  we have
	$
	|\xi|^k e^{-\widetilde{\epsilon}|\xi|^{1/s}}\leqslant \widetilde{ C}_{\epsilon}^{k} (k!)^s
	$
	for all $\xi\in\mathbb{Z}$ and $k\in\mathbb{N}$, where $\widetilde{C}_{\epsilon}=({s}/\widetilde{\epsilon})^s$.
	If $\alpha_j=0$ the result is immediate. For $\alpha_j>0$, according to Fa\`{a} Di Bruno's Formula (see Lemma 1.3.1 in \cite{KrPar02-book})  and the previous ine\-qua\-li\-ties we have
	\begin{eqnarray*}
		e^{-\widetilde{\epsilon}|\xi|^{1/s}}\!\left| \partial_{t_j}^{\alpha_j} e^{i\xi A_j(t_j)}\right|\!\! &=& \! e^{-\widetilde{\epsilon}|\xi|^{1/s}}\! \left|\sum_{\Delta(\alpha_j)}\dfrac{\alpha_j!}{k_1!\ldots k_{\alpha_j}!}e^{i\xi A_j(t_j)}\! \prod_{\ell=1}^{\alpha_j}\!\left[\dfrac{\partial_{t_j}^{\ell} (i\xi A_j(t_j))}{\ell !}\right]^{k_\ell}\right|\\
		&\leqslant & \!   e^{-\widetilde{\epsilon}|\xi|^{1/s}}\sum_{\Delta(\alpha_j)}\dfrac{\alpha_j!}{k_1!\ldots k_{\alpha_j}!}|\xi|^k \prod_{\ell=1}^{\alpha_j}\left[\dfrac{C_jh_j^{\ell}(\ell!)^s}{\ell!}\right]^{k_\ell}\\
		&\leqslant &  \! \sum_{\Delta(\alpha_j)}\dfrac{\alpha_j!}{k_1!\ldots k_{\alpha_j}!}(\widetilde{C}_{\epsilon} C_j)^k (k!)^s h_j^{\alpha_j}\prod_{\ell=1}^{\alpha_j}(\ell!)^{(s-1)k_{\ell}},
	\end{eqnarray*}
	where $\Delta(\alpha_j)$ is defined in \eqref{delta definition} and  $k\doteq k_1+\cdots+k_{\alpha_j}$.

	Finally, by applying Lemmas  \ref{lema_estimativa_produtorio_vzs_fatorial} and \ref{lema_sominha_maluca}  we obtain
	\begin{eqnarray*}
		e^{-\widetilde{\epsilon}|\xi|^{1/s}}\left| \partial_{t_j}^{\alpha_j} e^{i\xi A_j(t_j)}\right|&\leqslant &  h_j^{\alpha_j}\sum_{\Delta(\alpha_j)}\dfrac{\alpha_j!}{k_1!\ldots k_{\alpha_j}!}(\widetilde{C}_{\epsilon} C_j)^k k! (\alpha_j!)^{(s-1)}\\
		&\leqslant &  h_j^{\alpha_j}(\alpha_j!)^s\sum_{\Delta(\alpha_j)}\dfrac{k!}{k_1!\ldots k_{\alpha_j}!}(\widetilde{C}_{\epsilon} C_j)^k \\
		&\leqslant &  h_j^{\alpha_j}(\alpha_j!)^s (\widetilde{C}_{\epsilon} C_j)(1+\widetilde{C}_{\epsilon} C_j)^{\alpha_j-1} \\\\
		&\leqslant &  C_{\epsilon,j}^{\alpha_j} ({\alpha_j} !)^s,
	\end{eqnarray*}
	for all $(t_j,\xi)\in\mathbb{T}^{1}\times \mathbb{Z}$ and $\alpha_j\in\mathbb{N}_0,$ where $C_{\epsilon,j} = h_j(1+\widetilde{C}_{\epsilon} C_j)$.
	
	By taking $C_{\epsilon}=\max\{C_{\epsilon,j}:\; j=1,\ldots , n\}$ the proof is complete.

\end{proof}

\begin{proof31}
	Given $\epsilon > 0$ and $h > 0$, we have to show that there exists $C=C_{\epsilon, h}  > 0$ such that
	$$
	|\langle \widehat{u}(t,\xi)e^{i\xi A(t)}, \varphi \rangle| \leqslant  C \| \varphi \|_{s, h} e^{\epsilon |\xi|^{{1}/{s}}},
	$$
	for all $\xi \in \mathbb{Z}$ and $\varphi \in G^{s, h} (\mathbb{T}^{n})$.
	
	Consider $\varphi \in G^{s, h} (\mathbb{T}^{n})$. It follows from Lemma \ref{important lemma} that there exists $C_{\epsilon} > 0$ such that
	\begin{eqnarray}\label{ineq 1}
	\displaystyle |e^{-\frac{\epsilon}{2} |\xi|^{{1}/{s}}} \partial_t^{\alpha}(e^{i \xi A(t)}  \varphi(t))| &=&
	\left|\displaystyle \sum_{\beta \leqslant  \alpha} \binom{\alpha}{\beta} e^{-\frac{\epsilon}{2} |\xi|^{1/s}}( \partial_t^{\beta}e^{i \xi A(t)})(  \partial_t^{\alpha - \beta} \varphi(t))\right| \nonumber\\
	&\leqslant & \displaystyle\sum_{\beta \leqslant  \alpha} \binom{\alpha}{\beta} C_{\epsilon}^{|\beta|}(\beta!)^{s}
	\| \varphi \|_{s, h} ((\alpha - \beta)!)^s h^{|\alpha - \beta|} \nonumber\\
	&\leqslant &
	\widetilde{{C}}^{|\alpha|}_{\epsilon, h}(\alpha!)^s\| \varphi \|_{s, h},
	\end{eqnarray}
	where $\widetilde{C}_{\epsilon, h} = 2(h + 1)(C_{\epsilon}+1)$.
	
	By taking $\overline{C}_{\epsilon}=({2s}/{\epsilon})^s$ we have
	\begin{eqnarray}\label{ineq 2}
	|\xi|^{k} e^{-\frac{\epsilon}{2} |\xi|^{{1}/{s}}} \leqslant  \overline{C}_{\epsilon}^{k} (k!)^s,
	\end{eqnarray}
	for all $\xi\in\mathbb{Z}$ and $k\in\mathbb{N}_0$.
	
	Setting $\eta^{-1} \doteq \max \{\widetilde{C}_{\epsilon, h}, \overline{C}_{\epsilon} \}$, by the continuity of $u$ there exists $C_\eta > 0$ such that
	$$
	|\langle u, \psi \rangle| \leqslant  C_\eta \sup_{(t,x)\in\mathbb{T}^{n+1}} \sup_{(\alpha, k)\in\mathbb{N}_{0}^{n+1}} |\partial^{(\alpha, k)}\psi(t, x)| \eta^{|\alpha|+ k}(\alpha! k!)^{-s},
	$$
	for all $\psi \in G^s(\mathbb{T}^{n+1}).$ Therefore,
	\begin{align*}
	|\langle \widehat{u}(t,\xi) e^{i \xi A(t)}, \varphi(t) \rangle| = |\langle \widehat{u}(t,\xi), e^{i \xi A(t)} \varphi(t) \rangle| = \dfrac{1}{2\pi}|\langle u, e^{i \xi A(t)} \varphi(t) e^{-i \xi x} \rangle|  \\
	\leqslant  \dfrac{C_\eta}{2\pi} \,\, \sup_{(t,x)\in\mathbb{T}^{n+1}} \sup_{(\alpha, k)\in\mathbb{N}_{0}^{n+1}} |\partial^{\alpha}_t ( e^{i \xi A(t)} \varphi(t)) \partial^{k}_x  e^{-i \xi x}|
	\eta^{|\alpha|+ k}(\alpha! k!)^{-s}.
	\end{align*}
	
	On the other hand, it follows from \eqref{ineq 1} and \eqref{ineq 2} that
	\begin{align*}
	&|\partial^{\alpha}_t (e^{i \xi A(t)} \varphi(t)) | \,\,
	|\partial^k_x e^{-i \xi x}| \,\,
	\eta^{|\alpha|+ k}(\alpha! k!)^{-s} \\[2mm]
	&= |e^{-\frac{\epsilon}{2} |\xi|^{{1}/{s}}} \partial^{\alpha}_t (e^{i \xi A(t)} \varphi(t))| \,\,
	|\xi|^k e^{-\frac{\epsilon}{2} |\xi|^{{1}/{s}}} \,\,
	\eta^{|\alpha|+ k}(\alpha! k!)^{-s} e^{\epsilon |\xi|^{{{1}/{s}}}}  \\[2mm]
	&\leqslant  \widetilde{C}^{|\alpha|}_{\epsilon, h}(\alpha!)^s\| \varphi \|_{s, h}
	\overline{C}_{\epsilon}^{k}(k!)^s
	\eta^{|\alpha|+ k}(\alpha! k!)^{-s} e^{\epsilon |\xi|^{ { 1 }/{ s } } } \\[2mm]
	&\leqslant  \max \{\widetilde{C}_{\epsilon, h}, \overline{C}_{\epsilon} \}^{|\alpha| + k} \,\, \eta^{|\alpha| + k}
	\| \varphi \|_{s, h}e^{\epsilon |\xi|^{ { 1 }/{ s } } } \\[2mm]
	& =\; \| \varphi \|_{s, h}e^{\epsilon |\xi|^{ { 1 }/{ s } } },
	\end{align*}
	for all $(t,x)\in\mathbb{T}^{n+1}$ and $(\alpha, k)\in\mathbb{N}_{0}^{n+1}.$
	
	The previous inequalities imply	that
	$$
	|\langle \widehat{u}(t,\xi) e^{i \xi A(t)}, \varphi(t) \rangle| \leqslant  \dfrac{C_\eta}{2\pi} \| \varphi \|_{s, h}
	e^{\epsilon |\xi|^{ { 1 }/{ s } } },
	$$
	thereby obtaining  $Tu \in D'_{s}(\mathbb{T}^{n+1})$.
	
	Analogously,  $T^{-1} : D'_{s}(\mathbb{T}^{n+1}) \to D'_{s}(\mathbb{T}^{n+1})$ is well defined and it is given by
	$$
	u = \sum_{\xi \in \mathbb{Z}} \widehat{u}(t,\xi)e^{i \xi x} \mapsto T^{-1}u = \sum_{\xi \in \mathbb{Z}} \widehat{u}(t,\xi)e^{- i\xi A(t)}e^{i \xi x}, \
	u \in D'_{s} (\mathbb{T}^{n+1}).
	$$
	
	Finally, in order to prove that  $T$ is also an automorphism in ${G^s(\mathbb{T}^{n+1})}$, it is sufficient  to verify that $G^s(\mathbb{T}^{n+1})$ is $T-$invariant and $T^{-1}-$invariant.
	
	Let $u \in G^s(\mathbb{T}^{n+1})$. We must show that there exist $C > 0$, $h > 0$ and $\epsilon > 0$ such that
	$$
	|\partial_t^{\alpha} \big(\widehat{u}(t,\xi)e^{ i \xi A(t) }\big)| \leqslant  C h^{|\alpha|}(\alpha!)^{s}e^{-\epsilon |\xi|^{{ 1 }/{ s} }},
	$$
	for all $\alpha \in \mathbb{N}_0^n$ and $(t,\xi) \in\mathbb{T}^{n}\times \mathbb{Z}.$

	If $u \in G^s(\mathbb{T}^{n+1})$, then there exist $C_0 > 0$, $h_0 > 0$ and $\epsilon_0 > 0$ such that
	$$
	|\partial_t^{\beta} \widehat{u}(t,\xi)| \leqslant  C_0 h_0^{|\beta|}(\beta!)^{s}e^{-\epsilon_0 |\xi|^{{ 1 }/{ s} }},
	$$
	for all $\beta \in \mathbb{N}_0^n$ and $(t,\xi) \in\mathbb{T}^{n}\times \mathbb{Z}.$
	
	Therefore,
	\begin{align*}
	|\partial_t^{\alpha} \big(\widehat{u}(t,\xi)e^{ i \xi A(t) }\big)| &\leqslant
	\sum_{\beta \leqslant  \alpha} \binom{\alpha}{\beta} |\partial^{\beta}\widehat{u}(t,\xi) | |\partial^{ \alpha - \beta} e^{ i \xi A(t) }|  \\
	&\leqslant  \sum_{\beta \leqslant  \alpha} \binom{\alpha}{\beta}  C_0 h_0^{|\beta|}(\beta!)^{s}e^{-({\epsilon_0}/{2}) |\xi|^{{{ 1 }/{ s}} }}
	|\partial^{ \alpha - \beta} e^{ i \xi A(t) }| e^{-({\epsilon_0}/{2}) |\xi|^{{{ 1 }/{ s}} }} \\
	&\leqslant  \sum_{\beta \leqslant  \alpha} \binom{\alpha}{\beta}  C_0 h_0^{|\beta|}(\beta!)^{s}e^{-({\epsilon_0}/{2}) |\xi|^{{{ 1 }/{ s}} }}
	C_{\epsilon}^{|\alpha - \beta|} ((\alpha - \beta)!)^s  \\
	&\leqslant  C_0 (2 h_0 C_\epsilon)^{|\alpha|} (\alpha!)^s e^{-({\epsilon_0}/{2}) |\xi|^{{{ 1 }/{ s}} }},
	\end{align*}
	for all $\alpha \in \mathbb{N}_0^n$ and $(t,\xi) \in\mathbb{T}^{n}\times \mathbb{Z}.$
	
	Analogously, $G^s(\mathbb{T}^{n+1})$ is $T^{-1}-$invariant. Also,  the following conjugation is an immediate consequence
	$$
	(\widetilde{L}_1, \dots, \widetilde{L}_n) = (T L_1 T^{-1}, \dots, T L_n T^{-1}).
	$$
	
\end{proof31}

\section{Sufficiency in Theorem \ref{main theorem}} \label{section-sufficiency}

Let $u \in D'_{s} (\mathbb{T}^{n+1})$ such that $L_k u = f_k \in G^s(\mathbb{T}^{n+1})$, for all $k = 1, \dots, n$. First, we will show that $u \in G^s(\mathbb{T}^{n+1})$ provided that there is  $b_j\not\equiv0$ such that $t_j \in \mathbb{T}^1 \mapsto b_j(t_j)$ does not change sign.

We may assume that $b_j(t_j)\leqslant  0$ and thus $b_{j0}<0$ (if $b_j(t_j)\geqslant  0$ the arguments are similar). Moreover, according to the previous section, we may assume that
$$
L_j = \dfrac{\partial}{\partial {t_j} } + (a_{j0} + ib_j(t_j))\dfrac{\partial}{\partial x }.
$$

Consider the formal $x-$Fourier series of $u$ and $f_j$ given by
$$
u(t, x) = \sum_{\xi \in \mathbb{Z}} \widehat{u}(t,\xi) e^{i \xi x}
\text{ \ and \ }
f_j(t, x) = \sum_{\xi \in \mathbb{Z}} \widehat{f}_{j}(t,\xi) e^{i \xi x}.
$$

Replacing these series in the equation $L_j u = f_j$ we obtain, for each $\xi\in\mathbb{Z}$, the ordinary differential equation
\begin{equation*}
(\partial_{t_j} + i\xi (a_{j0} + ib_j(t_j)) \widehat{u}(t,\xi) = \widehat{f}_j(t,\xi),
\end{equation*}
with $t\in\mathbb{T}^n$.

Since $b_{j0}\neq 0$, then for each $\xi \neq 0$  the equation above has a unique solution, which can be written in the following two equivalent ways:
\begin{equation}\label{solucao_eta_positivo}
\widehat{u}(t,\xi) = \dfrac{ 1 }{ 1 - e^{-i2\pi\xi c_{j0}} }\int_{0}^{2\pi} e^{-i\xi H(t_j, \tau)} \widehat{f}_j(t_1,\ldots,t_j-\tau,\ldots, t_n,\xi) d\tau,
\end{equation}
and
\begin{equation}\label{solucao_eta_negativo}
\widehat{u}(t,\xi) = \dfrac{ 1 }{ e^{i2\pi\xi c_{j0}} - 1 }\int_{0}^{2\pi} e^{ i\xi \widetilde{H}(t_j, \tau)} \widehat{f}_j(t_1,\ldots,t_j+\tau,\ldots, t_n,\xi) d\tau,
\end{equation}
where $c_{j0}=a_{j0}+ib_{j0}$,
$$
H(t_j, \tau) = a_{j0}\tau+i\int_{t_j - \tau}^{t_j} b_j(s) ds \quad \text{and} \quad \widetilde{H}(t_j, \tau) = a_{j0}\tau+i\int_{t_j}^{t_j+\tau} b_j(s) ds.
$$

To analyze the behaviour of the coefficients $\hat{u}(t,\xi)$, it is  convenient to choose the formula \eqref{solucao_eta_positivo} when $\xi\geqslant 1$, and the formula \eqref{solucao_eta_negativo} when $\xi\leqslant  -1$.

Let us start with the case $\xi\geqslant  1$. We remember that $b_{j0}<0$, thus it is possible to take $C_0>0$ such that
$$
|(1 - e^{-i2\pi\xi c_{j0}})^{-1}|\leqslant  C_0,  \mbox{ for } \xi \geqslant  1.
$$

Since $f_j\in G^{s}(\mathbb{T}^{n+1})$, there exist positive constants $C_1$, $h_1$ and $\epsilon$ such that
$$
|\partial_{t}^{\beta}\hat{f}_j(t,\xi)|\leqslant  C_1 h_1^{|\beta|}\beta!^{s}e^{-\epsilon |\xi|^{1/s}},
$$
for all $\beta \in\mathbb{N}_0^{n}$ and $(t,\xi)\in\mathbb{T}^{n}\times \mathbb{Z}.$

Observe that $\Im H(t_j, \tau)\leqslant  0$ for all $t_j,\tau\in[0,2\pi]$, since $b_{j}(t_j)\leqslant  0$.
Thus, by proceeding as in the proof of Lemma \ref{important lemma}, we obtain a constant $C_\epsilon>0$ such that
$$
e^{-(\epsilon/2)|\xi|^{1/s}}|\partial_{t_j}^{{m}} e^{-i\xi H(t_j, \tau)} |\leqslant  C_\epsilon^{{m}}({m}!)^{s},
$$
for all $t_j,\tau\in[0,2\pi],$ $m \in\mathbb{N}_0,$ and $\xi\geqslant  1$.

Now, if $\alpha=(\alpha_1,\ldots,\alpha_n)\in\mathbb{N}_0^{n}$, we write  $\beta=\alpha-\alpha_je_j$ and have
the following estimate to $|\partial_{t}^{\alpha} \widehat{u}(t,\xi)|$:
\begin{align*}
&
\left| ({ 1 - e^{-i2\pi\xi c_{j0}} })^{-1} \int_{0}^{2\pi} \partial_{t_j}^{\alpha_j}\left( e^{-i\xi H(t_j, r)}
\partial_t^{\beta}\widehat{f}_j(t_1,\ldots,t_j-r,\ldots, t_n,\xi) \right) dr \right| 				\\
&\leqslant  C_0 \int_{0}^{2\pi} \sum_{m =0}^{ \alpha_j} \binom{\alpha_j}{m} |\partial^{{m}}_{t_j} e^{-i\xi H(t_j, r)}|
|\partial_{t_j}^{\alpha_j - {m}} ( \partial_t^{\beta}\widehat{f}_{j}(t_j,\ldots,t_j-r,\ldots, t_n,\xi) )| dr \\
&\leqslant  C_0 \int_{0}^{2\pi} \sum_{m=0}^{\alpha_j} \binom{\alpha_j}{{m}}  C_\epsilon^{m} (m!)^s \; C_1h_1^{|\beta|+\alpha_j-m}(\alpha_j-m)!^{s}\beta!^{s}
e^{ -(\varepsilon/2) \xi^{{1}/{s}}} dr  \\
&\leqslant  C_0 C_1 (h_1+1)^{|\alpha|} (\alpha!)^s e^{ -({ \varepsilon }/{ 2 }) \xi^{ {1}/{s}} }
\int_{0}^{2\pi} \sum_{{m=0}}^{ \alpha_j} \binom{\alpha_j}{{m}} C_\epsilon^m dr\\
&\leqslant  2\pi C_0 C_1 [2 (C_\varepsilon+1) (h_1+1)]^{|\alpha|} (\alpha!)^s e^{ -({ \varepsilon }/{ 2 }) \xi^{{1}/{s}} },
\end{align*}
for all $t \in\mathbb{T}^{n},$ and $\xi\geqslant  1.$

When $\xi\leqslant  - 1$, by considering formula \eqref{solucao_eta_negativo}, we obtain a similar estimate to $\partial_t^{\alpha}\widehat{u}(t,\xi)$ as shown above. Also, for $\xi=0$ we have $\partial_{t_k}\hat{u}(t,0)=\hat{f}_k(t,0)\in G^{s}(\mathbb{T}_t^n)$, for all $k=1,\ldots, n,$ which implies that $\hat{u}(t,0)\in G^{s}(\mathbb{T}_t^n)$.

Therefore
$$
u(t, x) = \sum_{\xi\in\mathbb{Z}} \widehat{u} (t,\xi) e^{i\xi x} \in G^{s}(\mathbb{T}^{n+1}),
$$
and the sufficiency of condition (I) in  Theorem \ref{main theorem} is verified.
\vspace{0,5cm}

Suppose again   $u \in D'_{s} (\mathbb{T}^{n+1})$ and $L_j u = f_j \in G^s(\mathbb{T}^{n+1})$, for all $j = 1, \dots, n$. Additionally, assume that $J = \{j_1< \dots< j_\ell \}\neq \emptyset$ and the vector $a_{J0} = (a_{j_1 0}, \dots, a_{j_\ell 0})$ is neither rational nor exponential Liouville of order $s>1$.

By making changes of coordinates on the torus $\mathbb{T}^{n+1}$,  without loss of generality, we may assume that $J=\{1,\ldots,\ell\}$. Thus, we write the coordinates of $\mathbb{T}^{n+1}$ as $(t, x) = (t', t'', x)$ where $t'=(t_1,\ldots,t_\ell)$ and $t'=(t_{\ell+1},\ldots,t_n)$.

Consider the formal Fourier series of  $u$ and $f_j$ with respect to  the variables $(t',x)$ given by
\begin{eqnarray}
u(t,x) &=& \sum_{(\eta,\xi) \in \mathbb{Z}^{\ell} \times \mathbb{Z}} \widehat{u}(t'',\eta,\xi) e^{i( \eta\cdot t'+\xi x)}, \label{Fourier coefficients u}
\end{eqnarray}
and
\begin{eqnarray}
f_j (t,x) &=& \sum_{(\eta,\xi) \in \mathbb{Z}^{\ell} \times \mathbb{Z}} \widehat{f}_j(t'',\eta,\xi) e^{i( \eta\cdot t'+\xi x)}, \ j = 1, \dots, n. \label{Fourier coefficients fj}
\end{eqnarray}

Thanks to Section \ref{section-normal-form}, we may consider
$
L_j=\partial_{t_j}+a_{j0}\partial_x,\;j=1,\ldots, \ell,
$
then replacing the series \eqref{Fourier coefficients u} and \eqref{Fourier coefficients fj} in the equations $L_ju = f_j$,  $j=1,\ldots, \ell$, we obtain, for each  $(\eta,\xi)$ the following
$$
i(\eta_j + \xi a_{j0}) \widehat{u}(t'',\eta,\xi) = \widehat{f}_{j}(t'',\eta,\xi), \quad j=1, \dots, \ell.
$$

Since $a_{J0}$ is not rational, we consider for each $(\eta,\xi)\neq (0,0)$
$$
|\xi a_{M0}+\eta_M| =   \max_{1\leqslant  j\leqslant  \ell }|\xi a_{j0}+\eta_j|\neq 0.
$$

Therefore
\begin{equation}\label{solucao_para_bj_nulas_divido}
\widehat{u}(t'',\eta,\xi) = -i(\xi a_{M0}+\eta_M)^{-1} \widehat{f}_M(t'',\eta,\xi).
\end{equation}

On the other hand, there exist positive constants $C$, $h$ and $\epsilon$ such that
$$
|\partial_{t''}^{\alpha}\hat{f}_j(t'',\eta,\xi)|\leqslant  C h^{|\alpha|} \alpha!^s e^{-\epsilon|(\eta,\xi)|^{1/s}},
$$
for all $ t''\in\mathbb{T}^{n-\ell}$, $\alpha\in\mathbb{N}_0^{n-\ell}$, $(\eta,\xi)\in\mathbb{Z}^{\ell+1}$ and $ j=1,\ldots,\ell.$

Also, since $a_{J0}$ is not an exponential Liouville vector of order $s$, there exists $C_\epsilon >0$ such that
$$
\max_{1\leqslant j \leqslant \ell}|\xi a_{j0}+\eta_j|\geqslant  C_\epsilon e^{-(\epsilon/2) |\xi|^{1/s}},
$$
for all $(\eta,\xi)\in\mathbb{Z}^{n+1}$.

By replacing the previous estimates in \eqref{solucao_para_bj_nulas_divido} we obtain
\begin{eqnarray}\label{estimate not zero}
|\partial_{t''}^{\alpha}\widehat{u}(t'',\eta,\xi)|\leqslant  C_\epsilon^{-1}Ch^{|\alpha|}(\alpha!)^s e^{-(\epsilon/2)|(\eta,\xi)|^{1/s}},
\end{eqnarray}
for all $ t''\in\mathbb{T}^{n-\ell}$, $\alpha\in\mathbb{N}_0^{n-\ell}$ and $(\eta,\xi)\in\mathbb{Z}^{\ell+1}\setminus\{(0,0)\}$.

If $\ell=n$, then we take the Fourier coefficients in the series \eqref{Fourier coefficients u} with respect to the variables $(t,x)$ and, in this case, the inequality \eqref{estimate not zero} reduces to
$$
|\widehat{u}(\eta,\xi)|\leqslant  C_\epsilon^{-1}C e^{-(\epsilon/2)|(\eta,\xi)|^{1/s}},
$$
for all $(\eta,\xi)\in\mathbb{Z}^{n+1}$, by taking a larger $C$ if necessary. Thus, $u\in G^{s}(\mathbb{T}^{n+1})$.

If $1\leqslant \ell<n$, it follows from \eqref{estimate not zero} that
$$
v\doteq u-\hat{u}(t'',0,0)\in G^{s}(\mathbb{T}^{n+1}),
$$
where $u$ is given by \eqref{Fourier coefficients u}.

Thus, for each $j=\ell+1\ldots, n$ we have
$$L_j v = L_j u - \partial_{t_j}\hat{u}(t'',0,0)=f_j-\partial_{t_j} \hat{u}(t'',0,0).$$

Then $\partial_{t_j} \hat{u}(t'',0,0) = f_j-L_j v\in G^{s}(\mathbb{T}^{n+1})$, which implies that $\hat{u}(t'',0,0)\in G^{s}(\mathbb{T}_{t''}^{n})$.

Therefore, we may conclude that the inequality \eqref{estimate not zero} holds for all $(\eta,\xi)\in\mathbb{Z}^{n+1}$ and the proof of the sufficiency is complete.

\section{Necessity in Theorem \ref{main theorem}} \label{section-necessity}

The goal of this section is to present singular solutions to the system \eqref{system-main-thm} when the conditions (I) and (II) in Theorem \ref{main theorem} fail.

We split the proof in two cases depending on the existence or not of real vector fields in the system \eqref{system-main-thm}.

\subsection{System \eqref{system-main-thm} has no real vector fields ($J=\emptyset$)}\label{section no real vector fields}

We argue by contradiction. Assume that $b_j \not\equiv 0$ changes sign for all $j=1,\ldots, n$. In this situation, it is well known that the operators
$$L_j=\partial_{t_j}+(a_{j0}+ib_j(t_j))\partial_x,$$
are not globally $s-$hypoelliptic on $\mathbb{T}_{t_j,x}^2$ (see Proposition 1.3 in \cite{GPY93}).
This means that it is possible to find $u_j \in D'_s(\mathbb{T}^2)\setminus G^s(\mathbb{T}^2)$ such that $L_j u \in G^s(\mathbb{T}^2)$, for each $j=1,\ldots,n$.

The problem here is how to obtain a solution $u(t_1,\ldots,t_n, x) \in D'_s(\mathbb{T}^{n+1})\setminus G^s(\mathbb{T}^{n+1})$ for the  system \eqref{system-main-thm} from these singular solutions $u_j(t_j,x)$.

In Propositions \ref{lema_sol_singular_campo_em_R2_c0_racional} and \ref{lema_sol_singular_campo_em_R2_c0_irracional} we construct special singular solutions $u_j\in D_s'(\mathbb{T}^2)\setminus G^s(\mathbb{T}^2)$ for the vector fields $L_j$. Next, from these solutions, we obtain a singular solution to the system \eqref{system-main-thm}.

\begin{proposition}\label{lema_sol_singular_campo_em_R2_c0_racional}
	Consider the vector field
	$$L=\partial_t + (a_0 + ib(t))\partial_x,$$
	defined on $\mathbb{T}^2,$ where $a_0=p/q\in\mathbb{Q}$ and $b \in G^s(\mathbb{T}^1_t)$, $s>1$,  is a real-valued function with average $b_0=0$. Then there is
	$$
	u(t,x) = \sum_{k\in\mathbb{N}} \widehat{u}(t,qk) e^{i qk x} \in D'_{s}(\mathbb{T}^2),
	$$
	satisfying:
	\begin{enumerate}
		\item[$(i)$] $Lu = 0;$
		\item[$(ii)$] there is $t_0\in\mathbb{T}^1$ such that $|\widehat{u}(t_0,qk)| = 1$, for all $k\in\mathbb{N}$;
		\item[$(iii)$] for every $\epsilon > 0$,   there exist $C > 0$ and $h > 0$ such that
		$$
		|\partial^{\alpha}_{t} \widehat{u} (t,qk)|e^{-\epsilon |qk|^{{1}/{s} }} \leqslant  C h^{\alpha} (\alpha!)^s,
		$$
		for all $t \in \mathbb{T}^1,$ $\alpha  \in \mathbb{N}_0$ and $k \in \mathbb{N}$.
	\end{enumerate}
	In particular, $u \in D'_{s}(\mathbb{T}^2) \setminus G^s(\mathbb{T}^2)$ is a singular solution of $L$.
\end{proposition}

\begin{proof}	
	For each $k\in\mathbb{N}$ define
	$$\widehat{u}(t,qk)=e^{-iqk a_0 t}e^{qk (B(t)-B(t_0))}, \ t \in \mathbb{T}^1,$$
	where $B\in  G^s(\mathbb{T}^1)$ is such that $ B'=b$ and $B(t_0)=\max\{B(t):\; t\in \mathbb{T}^1\}$.
	
	The properties $(i)$ and $(ii)$ are immediate and the conclusion of $(iii)$ follows the same ideas of the proof of Lemma \ref{important lemma}.
	
	%
	
\end{proof}

\begin{proposition}\label{lema_sol_singular_campo_em_R2_c0_irracional}
	Consider the vector field
	$$L=\partial_t + (a_0 + ib(t))\partial_x,$$
	defined on $\mathbb{T}^2$, where $b\in G^s(\mathbb{T}^1_t)$, $s>1$, is a nonzero real-valued function and $a_0\in\mathbb{R}$. If $b$ changes sign and  $ a_0 + ib_0 \notin \mathbb{Q}$, then there is
	$$
	u(t,x) = \sum_{\xi\in\mathbb{N}} \widehat{u}(t,\xi)e^{i \xi x} \in D'_{s}(\mathbb{T}^2),
	$$
	satisfying:
	\begin{enumerate}
		\item[$(i)$] $Lu \in G^s(\mathbb{T}^2)$;
		
		\item[$(ii)$] there exist $t_0 \in \mathbb{T}^1$ and $C > 0$ such that, for any $\xi \in \mathbb{N}$,
		$$
		|\widehat{u}(t_0,\xi)| \geqslant  \dfrac{ C }{ \sqrt{\xi} };
		$$
		
		\item[$(iii)$] for every $\epsilon > 0$  there exist $C > 0$ and $h > 0$ such that
		$$
		|\partial^{\alpha}_{t} \widehat{u}(t,\xi)|e^{-\epsilon |\xi|^{ {1}/{s} }} \leqslant  C h^{\alpha} (\alpha!)^s,
		$$
		for all $t \in \mathbb{T}^1,$ $\alpha \in \mathbb{N}_0$ and $\xi \in \mathbb{N}.$
	\end{enumerate}
	In particular, $u \in D'_{s}(\mathbb{T}^2) \setminus G^s(\mathbb{T}^2)$ is a singular solution of $L$.
\end{proposition}

\begin{proof}
	We start by considering $b_{0}\leqslant 0$. In this case, we define the function
	$$
	H(t, r) =
	a_0r + i\int_{t-r}^{t} b(y) dy,\quad  0\leqslant t,r\leqslant  2\pi,
	$$
	and set
	$$
	B_0 =
	\max_{t, r \in [0, 2\pi]} \int_{t-r}^{t}b(y)dy = \int_{t_0-r_0}^{t_0} b(y) dy.
	$$
	Since $b(t)$ changes sign, we have $B_0 > 0$ and $r_0 \in (0, 2\pi)$. By performing a translation in the variable $t$ (if necessary), we may assume $0< t_0, r_0, t_0 - r_0 < 2\pi$.
	
	Let $\delta >0$ such that $I_\delta=[t_0 - r_0 - \delta, t_0 - r_0 + \delta] \subset (0, 2\pi)$ and consider $\varphi \in G_c^s((0, 2\pi))$ such that
	$\mbox{supp}(\varphi)$ is a compact subset of $I_\delta$, $\varphi(t) = 1$, for all $t \in [t_0 - r_0 - \frac{\delta}{2}, t_0 - r_0 + \frac{\delta}{2}]$, and $0 \leqslant  \varphi(t) \leqslant  1$, for all $t$.
	
	Consider  the formal series
	$$
	f(t,x) = \sum_{\xi \in\mathbb{N}} \widehat{f}(t,\xi) e^{i \xi x},
	$$
	where, for each $\xi\in\mathbb{N}$, the $x-$Fourier coefficient $\widehat{f}(t,\xi)$ is the periodic extension of the function
	$$
	(1 - e^{-i2\pi\xi c_0})e^{-B_0 \xi}  e^{-i \xi a_0 (t - t_0)}\varphi(t), \quad t\in [0, 2\pi].
	$$
	
	Since $b_0\leqslant 0$, we have $|(1 - e^{-i2\pi\xi c_0})|\leqslant 2$, for all $\xi\in \mathbb{N}$.
	
	If $a_0=0$ then immediately we have $f\in G^s(\mathbb{T}^2)$. Otherwise, let $C >0$ and $h>1$ be constants such that $ | \varphi^{(k)} (t)| \leq C h^{\alpha}(k!)^{s}$, for all $t \in [0, 2\pi]$ and $k \in \mathbb{N}_0$. Also, by considering $C_0=\frac{2}{B_0}$ we obtain $\xi^ke^{-\frac{B_0}{2}\xi}\leqslant C_0^k k!$, for all $\xi\in\mathbb{N}$ and $k\in\mathbb{N}_0$.
	
	Given $\alpha \in \mathbb{N}_0$ we have
	\begin{align}\label{ineq 3}
	|\partial_{t}^{\alpha} \widehat{f} (t,\xi)| 
	&\leqslant 2 e^{-B_0 \xi }
	\sum_{k=0}^{ \alpha} \dfrac{\alpha!}{{k}! (\alpha - {k})!} |\partial_{t}^{\alpha - {k}} \varphi(t)| |\partial_{t}^{{k}}e^{-i \xi a_0 t}|  \nonumber\\
	&\leqslant 2 e^{-B_0 \xi }
	\sum_{{k} =0}^{ \alpha} \dfrac{\alpha!}{{k}! (\alpha - {k})!} C h^{\alpha - {k}} (\alpha - {k})!^s \xi^k |a_0|^{k}  \nonumber\\	
	&\leqslant 2 C e^{-\frac{B_0}{2} \xi}
	\sum_{{k} =0}^{ \alpha} \dfrac{\alpha!}{{k}!}  h^{\alpha - {k}} (\alpha - {k})!^{s-1} C_0^k k! |a_0|^{k}  \nonumber \\
	&\leq 2 C (\alpha !)^s e^{-\frac{B_0}{2} \xi}
	\sum_{{k} =0}^{ \alpha}  h^{\alpha - {k}}    C_0^k  |a_0|^{k}  \nonumber \\
	&\leqslant 2 C [2h(1+C_0|a_0|)]^{\alpha} (\alpha!)^s e^{-\frac{B_0}{2} \xi^{1/s}},
	\end{align}
	for all $t \in [0,2\pi]$ and  $\xi\in\mathbb{N}$. Thus $f\in  G^s(\mathbb{T}^2)$.
	
	From now on, we will obtain  $u\in D'_{s}(\mathbb{T}^2) \setminus G^s(\mathbb{T}^2)$ such that $Lu=f$ and, additionally, $u$ satisfying the other conditions desired.

	Since $c_0=a_0+ib_0\notin \mathbb{Q}$, by taking the $x-$Fourier series in the equation $Lu=f$, we define $u$ as follows
	$$	
	u(t,x) = \sum_{\xi \in \mathbb{N}} \widehat{u}(t,\xi)  e^{i\xi x},
	$$
	where the $x-$Fourier coefficients are given by
	\begin{eqnarray}\label{coefficient u necessity}
	\widehat{u}(t,\xi)& \doteq & \dfrac{1}{1 - e^{-i2\pi \xi c_0}} \int_{0}^{2\pi}e^{-i\xi H(t, r)} \widehat{f}(t-r,\xi) dr \nonumber \\
	&=& e^{-i\xi a_0(t-t_0)} \int_{0}^{2\pi} e^{\xi(\Im H(t, r)-B_0)} \varphi(t-r) dr, \quad t \in \mathbb{T}^1.
	\end{eqnarray}
	
	Therefore, by definition we have $Lu=f$ and
	$
	|\langle \widehat u(\cdot,\xi), \phi \rangle |\leqslant (2\pi)^2 \|\phi\|_{\infty}
	$ for all $\phi\in  G^s(\mathbb{T}^1)$,
	which implies that $u\in D'(\mathbb{T}^2)$, in particular $u\in D_s'(\mathbb{T}^2)$.
	
	Consider the function $\psi(r) =  \Im H(t_0, r)- B_0$. Note that
	\begin{align*}
	|\widehat{u} (t_0,\xi)| 
	= \int_{r_0 - \delta}^{r_0 + \delta} e^{\xi \psi(r)} \varphi (t_0 - r) dr
	\geqslant \int_{r_0 - \frac{\delta}{2}}^{r_0 + \frac{\delta}{2}} e^{\xi \psi(r)} dr.
	\end{align*}
	Since  $\psi(r_0) = 0$ and $\psi'(r_0) = b(t_0 - r_0) = 0$, it follows from Taylor's formula that, for each $h \in (-\delta, \delta)$ there exists  $\theta(h)$ suh that
	$$
	\psi(r_0 + h) = \psi''(r_0+\theta(h))\frac{h^2}{2}.
	$$
	Thus, by taking
	$$
	A = \sup \left\{\left|\frac{1}{2}\psi''(t) \right|; t \in [r_0 - \delta, r_0 + \delta] \right\},
	$$
	there exists $C>0$ such that
	\begin{align*}
	|\widehat{u}(t_0,\xi)|\geqslant
	\int_{- \frac{\delta}{2}}^{\frac{\delta}{2}} e^{\xi \psi(r_0 + h)} dh
	&\geqslant \int_{- \frac{\delta}{2}}^{\frac{\delta}{2}} e^{-\xi Ah^2} dh
	= \int_{- \frac{\delta\sqrt{A}}{2}}^{\frac{\delta\sqrt{A}}{2}} e^{-\xi w^2}dw  \geqslant \dfrac{C}{\sqrt{\xi}},
	\end{align*}
	for all $\xi\in \mathbb{N},$ which implies that $u \notin G^s(\mathbb{T}^2)$.\\
	
	Finally, for every $\epsilon>0$ and $\alpha\in\mathbb{N}_0$ we obtain from  \eqref{coefficient u necessity} that
	\begin{align*}
	& |\partial^{\alpha}_{t} \widehat{u}(t,\xi)|e^{-\epsilon \xi^{ {1}/{s} }} \\[2mm]
	& \leqslant  \int_{0}^{2\pi} \left|\partial^{\alpha}_{t} \big( e^{\xi(\Im H(t, r)-B_0)} e^{-i \xi a_0 t} \varphi(t-r) \big)\right|e^{-\epsilon \xi^{ {1}/{s} }} dr  \\
	& \leqslant  \int_{0}^{2\pi} \sum_{k=0}^{\alpha} \dfrac{\alpha!}{k!(\alpha - k)!}
	\left|\partial^{\alpha - k}_{t} e^{\xi(\Im H(t, r)- B_0)}\partial^{k}_{t} \big(e^{-i \xi a_0 t} \varphi(t-r)\big)\right| e^{-{\epsilon} \xi^{ {1}/{s} }} dr.
	\end{align*}
	
	Since $\Im H(t, r)-B_0 \leqslant 0$, then by a calculus similar to the one made in the proof of Lemma \ref{important lemma},  there exists $C_0>0$ and $h_0>0$ such that
	$$
	e^{-(\epsilon/2) \xi^{ {1}/{s} }} |\partial^{\alpha - k}_{t} e^{\xi(\Im H(t, r)-B_0)}|\leqslant C_0 h_0^{\alpha-k}(\alpha-k)!^s.
	$$
	Also, analogous ideas to those made in \eqref{ineq 3} imply that there exists $C_1>0$ and $h_1>0$ such that
	$$
	e^{-(\epsilon/2) \xi^{ {1}/{s} }}| \partial^{k}_{t} \big(e^{-i \xi a_0 t} \varphi(t-r)\big)| \leqslant C_1 h_1^k (k!)^s.
	$$
	
	Therefore,
	\begin{align*}
	|\partial^{\alpha}_{t} \widehat{u}(t,\xi)|e^{-\epsilon \xi^{ {1}/{s} }}
	&\leqslant  \int_{0}^{2\pi} \sum_{k=0}^{\alpha} \dfrac{\alpha!}{k!(\alpha - k)!}
	C_0C_1h_0^{\alpha-k}h_1^{k}(\alpha-k)!^s k!^s dr\\
	& \leqslant 2\pi C_0C_1 h_2^{\alpha} \sum_{k=0}^{\alpha} {\alpha!}(k!(\alpha - k)!)^{s-1}\\
	&\leqslant 2\pi C_0C_1(2h_2)^\alpha (\alpha!)^s, \quad \forall  t\in\mathbb{T}^1,\; \;\forall \xi\in\mathbb{N}_0,
	\end{align*}
	where $h_2=\max\{1,h_0,h_1\}$.\\

	If $b_0>0$, the arguments are closely similar to the previous one. In this case, we set
	$$
	\widetilde{H}(t, r) =  a_0r + i\int_{t}^{t+r} b(y) dy,\quad  0\leqslant t, r  \leqslant 2\pi,
	$$
	and
	$$
	B_0 =  \min_{t, r \in [0, 2\pi]} \int_{t}^{t+r} b(y) dy = \int_{t_0}^{t_0 + r_0} b(y) dy.
	$$
	
	In this situation,  $B_0 < 0$ and we may assume $r_0, t_0, t_0 + r_0 \in (0, 2\pi)$.
	
	We take $\delta >0$ such that $I_\delta=[t_0 + r_0 - \delta, t_0 + r_0 + \delta] \subset (0, 2\pi)$ and consider $\varphi \in G_c^s((0, 2\pi))$ such that $\mbox{supp}(\varphi)$ is a compact subset of $I_\delta$, $\varphi(t) = 1$, for all $t \in [t_0 + r_0 - \frac{\delta}{2}, t_0 + r_0 + \frac{\delta}{2}]$ and $0 \leqslant  \varphi(t) \leqslant  1$, for all $t$.
	
	Under the previous notations we define
	
	$$
	f(t,x) = \sum_{\xi \in\mathbb{N}} \widehat{f}(t,\xi) e^{i \xi x},
	$$
	where for each   $\xi\in\mathbb{N}$ the coefficient  $\widehat{f}(t,\xi)$ is the periodic extension of
	$$
	(e^{ i\xi2\pi c_0} - 1)e^{\xi B_0} \varphi(t) e^{i \xi a_0 (t - t_0)}, \quad t \in [0, 2\pi].
	$$
	
	We conclude this case proceeding as in the previous one.
	
\end{proof}

Now we are ready to construct a singular solution $u\in D_s'(\mathbb{T}^{n+1})\setminus G^s(\mathbb{T}^{n+1})$ to the system \eqref{system-main-thm}.
Without loss of generality we may suppose $c_{j0} \notin \mathbb{Q}$, for all $j \in \{1, \dots, m\}$, $c_{j0} = {p_j}/{q_j}$, $p_j \in \mathbb{Z}$ and $q_j \in \mathbb{N}$ for
all $j \in \{m+1, \dots, n\}$.

For each $j=1,\ldots,n$, we consider
$$
u_j (t_j, x) = \sum_{\xi \in \mathbb{N}} \widehat{u}_{j} (t_j,\xi) e^{i \xi x} \in D'_{s}(\mathbb{T}^2) \setminus G^s(\mathbb{T}^2),
$$
which satisfies the conditions of Proposition \ref{lema_sol_singular_campo_em_R2_c0_irracional}  when $j=1,\ldots, m$ and the conditions of
Proposition \ref{lema_sol_singular_campo_em_R2_c0_racional} when $j=m+1,\ldots, n$.

We define $u\in D_s'(\mathbb{T}^{n+1})$ as follows
$$
u(t,x) = \sum_{k \in \mathbb{N}} \widehat{u}(t,qk)e^{iqk x}= \sum_{k \in \mathbb{N}} \left(\prod_{j=1}^{n}\widehat{u}_{j}(t_j,qk)\right) e^{iqk x},
$$
where  $q $ is a positive multiple of  $q_{m+1},\ldots, q_n$.

Since there exists $t_0=(t_{10}, \dots, t_{n_0})$ such that
\begin{eqnarray}\label{ineq 4}
|\widehat{u} (t_{0},qk)| = \prod_{j=1}^{m}|\widehat{u}_{j} (t_{j0},qk)| \geqslant  \left(\dfrac{C}{\sqrt{qk}}\right)^m, \quad \forall  k\in\mathbb{N},
\end{eqnarray}
we have $u \notin G^s(\mathbb{T}^{n+1})$.

Note that for each $j = 1, \dots, n$
$$
L_j u(t,x) = \sum_{k\in \mathbb{N}}\left(\prod_{\ell=1, \ell\neq j}^{n} \widehat{u}_{\ell}(t_\ell,qk)\right)\widehat{L_j u_j}(t_j,qk)
e^{i qk x}.
$$
Then, by Proposition \ref{lema_sol_singular_campo_em_R2_c0_racional}, for each $j=m+1,\ldots, n$ we have $L_ju=0$.

If $j=1,\ldots, m$, by Proposition \ref{lema_sol_singular_campo_em_R2_c0_irracional} we have $f_j(t_j,x)\doteq L_ju_j\in G^{s}(\mathbb{T}^2)$ and, consequently,  there exist positive constants $ \epsilon $, $h_j$ and $C_j$  such that
$$
|\partial^{\alpha_j}_{t_j} \widehat{f}_{j} (t_j,\xi)| \leqslant  C_j h_j^{\alpha_j} (\alpha_j!)^s e^{-\epsilon |\xi|^{ {1}/{s} }},
$$
for all $t_j \in \mathbb{T}_{t_j}^{1}, \, \alpha_j \in \mathbb{N}_0 $ and $ \xi \in \mathbb{N}.$

Finally, given $j\in\{1,\ldots, m\}$ we have
\begin{align*}
|\partial^{\alpha}_t \widehat{L_j u} (t,qk )| &= \left(\prod_{\ell=1, \ell\neq j}^{n} |\partial^{\alpha_\ell}_{t_\ell} \widehat{u}_{\ell} (t_\ell,qk)|\right)|\partial^{\alpha_j}_{t_j} \widehat{f}_{j } (t_j,qk)|\\
&\leqslant \left(\prod_{\ell=1, \ell\neq j}^{n} |\partial^{\alpha_\ell}_{t_\ell} \widehat{u}_{\ell} (t_\ell,qk)| e^{-\frac{\epsilon}{n} |qk |^{ {1}/{s} }}\right) C_j h_j^{\alpha_j} (\alpha_j!)^s e^{-\frac{\epsilon}{n} |qk |^{ {1}/{s} }} \\
&\leqslant \left(\prod_{\ell=1, \ell\neq j}^{n}  C_\ell h_\ell^{\alpha_\ell} (\alpha_\ell!)^{s} \right) C_j h_j^{\alpha_j} (\alpha_j!)^s e^{-\frac{\epsilon}{n} |qk |^{ {1}/{s} }} \\
&\leqslant \left(\prod_{\ell=1}^{n}  C_\ell \right) h^{|\alpha|} (\alpha!)^{s} e^{-\frac{\epsilon}{n} |qk |^{ {1}/{s} }},
\end{align*}
for all $t \in \mathbb{T}^n, k\in\mathbb{N}$ and $\alpha \in \mathbb{N}_0^n,$ where $h=\max\{-h_1,\ldots, h_n,1\}$.

Therefore, we conclude that  $L_ju \in G^s(\mathbb{T}^{n+1})$, for all $j = 1, \dots, n$.

\subsection{System \eqref{system-main-thm}  has real vector fields ($J\neq\emptyset$)}

Reordering the coordinates  in $\mathbb{T}^{n+1}$ if necessary,  we may assume that $J = \left\{ j \in \{1, \dots, n\}: \;b_j(t_j) \equiv 0 \right\} = \{1, \dots, \ell\}$ for some $1\leqslant \ell \leqslant n$.

Thus, in this section we consider that
\begin{itemize}
	\item[$\circ$] $b_j \equiv 0$ for all
	$j = 1, \dots, \ell$;
	\item[$\circ$] $b_j$ changes sign for all $j = \ell+1, \dots, n$;
	\item[$\circ$] $a_{J0}=(a_{10},\ldots,a_{\ell0})$ is either rational or an exponential Liouville vector of order $s>1$.
\end{itemize}

We  write  $(t, x) = (t', t'', x)\in \mathbb{T}^{n+1}$, where $t'=(t_1,\ldots, t_\ell)$ and $t''=(t_{\ell+1},\ldots,t_n)$. Also, we denote  $a_0'=a_{J0}$ and $a_0''=(a_{(\ell+1)0},\ldots,a_{n0})$.

Suppose first that $a_0'\in\mathbb{Q}^{\ell}$. Therefore $J\subset I$, where $I=\{j\in\{1,\ldots, n\}:\; a_{j0}\in\mathbb{Q}\}$.

Let $q$ be a positive integer such that $qa_{j0}\in\mathbb{Z}$, for all $j\in I$. It follows from Subsection \ref{section no real vector fields} that there exists
$$
v(t'',x)=\sum_{k\in\mathbb{N}} \widehat{v}(t'',q k)e^{iqk x}\in D_s'(\mathbb{T}^{n-\ell+1})\setminus  G^s(\mathbb{T}^{n-\ell+1}),
$$
such that  $g_j(t'',x)\doteq L_j v\in  G^s(\mathbb{T}^{n-\ell+1})$ for all $j=\ell+1,\ldots, n$.

We define
$$
u(t,x)=\sum_{k\in\mathbb{N}} \widehat{u}(t,q k)e^{iqk x} = \sum_{k\in\mathbb{N}} \widehat{v}(t'',q k)e^{-i qka_0'\cdot t'} e^{iqk x}.
$$
Thus, $u\in D'(\mathbb{T}^{n+1})$ and  since
$$
|\widehat{u} (t',t''_0,qk)|=|\widehat{v} (t_0'',qk)|  \geqslant  \left(\dfrac{C}{\sqrt{qk}}\right)^m, \quad k\in\mathbb{N},
$$
where $t_0''$ can be defined as in \eqref{ineq 4} and $0\leqslant m\leqslant n-\ell$ is the amount of irrational averages $a_{j0}$, with $j\in \{\ell+1,\ldots,n\}\setminus I $, we conclude  that $u\notin  G^s(\mathbb{T}^{n+1})$.

On the other hand, for each $k\in\mathbb{N}$, we have
$$
\widehat{f}_j(t,qk)\doteq\widehat{L_j u}(t,qk)=
\begin{cases}
0, & j=1,\ldots, \ell\\
e^{-i qka_0'\cdot t'} \widehat{g}_j(t'',qk), & j=\ell+1,\ldots, n,
\end{cases}
$$
which implies that $f\doteq L_ju\in G^{s}(\mathbb{T}^{n+1})$, $j=1,\ldots, n$. Indeed, for every $\epsilon>0$ consider a constant $C_{\epsilon}>0$ such that $|\xi|^{m}e^{-\frac{\epsilon}{2} |\xi|^{1/s}}\leqslant C_\epsilon^m m!^{s}$ for all $\xi\in\mathbb{Z}$ and $m\in\mathbb{N}$. Also, for each $j=\ell+1,\ldots,n$,
there exist positive constants $C, h$ and $\epsilon$ such that
\begin{eqnarray*}
	|\partial_t^{\alpha}\widehat{f}_j(t,qk)|&=&|qk|^{|\alpha'|}|\partial_{t''}^{\alpha''}\widehat{g}_j(t'',qk)|\\[2mm]
	&\leqslant & |qk|^{|\alpha'|}C h^{|\alpha''|}(\alpha''!)^s e^{-\epsilon |qk|^{1/s}} \\[2mm]
	&\leqslant& C C_\epsilon^{|\alpha'|} (|\alpha'|!)^s  h^{|\alpha''|}(\alpha''!)^s e^{-\frac{\epsilon}{2} |qk|^{1/s}} \\[2mm]
	&\leqslant& C C_\epsilon^{|\alpha'|} (\ell^{s})^{|\alpha'|}(\alpha'!)^s  h^{|\alpha''|}(\alpha''!)^s e^{-\frac{\epsilon}{2} |qk|^{1/s}} \\[2mm]
	&\leqslant& C h_1^{|\alpha|}\alpha!^se^{-\frac{\epsilon}{2} |qk|^{1/s}},
\end{eqnarray*}
for all $t\in\mathbb{T}^n,$ $k\in \mathbb{N}$ and $\alpha=(\alpha',\alpha'')\in \mathbb{N}_0^n$, where $h_1=\max\{1, C_\epsilon \ell^s, h\}$.

We will assume from now on that $a_0'\notin \mathbb{Q}^{\ell}$ is an exponential Liouville vector of order $s>1$.

\begin{lemma}\label{lema_sequencia_esperta_exp_liouville_ordem_s}
	If $\alpha = (\alpha_1, \dots, \alpha_\ell)\notin \mathbb{Q}^{\ell}$ is an exponential Liouville vector of order $s \geqslant  1$ and $q \in \mathbb{N}$, then there are
	$\epsilon > 0$ and a sequence $(p_{k}, q_k)_{k \in \mathbb{N}}$ in $ (q\mathbb{Z})^{\ell} \times (q \mathbb{N})$ such that $q_k < q_{k+1}$ and
	$$
	\max_{1 \leqslant  j \leqslant  \ell} | p_k^{(j)}+q_k\alpha_j  | \leqslant q
	e^{- \epsilon q_k^{ { 1 }/{ s } } },
	$$
	for all $ k \in \mathbb{N}.$
\end{lemma}
\begin{proof}
	By Definition \ref{E^sL vector}, there exist $\delta > 0$ and a sequence $(r_{k}, s_{k})_{k \in \mathbb{N}} $ in $ \mathbb{Z}^n \times \mathbb{N}$ such that $s_k \to \infty$ and, for all $ k \in \mathbb{N}$,
	$$
	\max_{1 \leqslant  j \leqslant  \ell}| r_{k}^{(j)} + s_k\alpha_j | \leqslant e^{- \delta s_k^{ { 1 }/{ s } } }.
	$$
	
	By passing to a subsequence, if necessary,  we may assume that $(s_k)_{k \in \mathbb{N}}$ is strictly increasing. Let $\epsilon > 0$ such that $\epsilon q^{{1}/{s}} \leqslant  \delta$. Thus $-\delta s_k^{{1}/{s}} \leqslant - \epsilon ({q}s_k)^{{1}/{s}}$ and, by setting
	$(p_k,q_k) = ({q}r_k, {q}s_k)$, we obtain
	$$
	\max_{1 \leqslant  j \leqslant  \ell} | p_k^{(j)}+q_k\alpha_j  |\leqslant q  e^{- \delta s_k^{{ 1 }/{ s } } }\leqslant qe^{- \epsilon (q_k)^{{1}/{s}}},
	$$
	for all $k \in \mathbb{N}$.
	
\end{proof}

Let $q$ be a positive integer such that $qa_{j0}\in\mathbb{Z}$ for all $j\in I=\{j: a_{j0}\in\mathbb{Q}\} $. By  Subsection \ref{section no real vector fields} there exists  $v(t'',x)\in D_s'(\mathbb{T}^{n-\ell+1})\setminus  G^s(\mathbb{T}^{n-\ell+1})$
such that  $g_j(t'',x)\doteq L_j v\in  G^s(\mathbb{T}^{n-\ell+1})$ where for each coefficient $\widehat{v}(t'',\xi)\neq 0$ we have $\xi\in q\mathbb{N}$.

By taking the sequence $(p_{k}, q_k)$ obtained in Lemma \ref{lema_sequencia_esperta_exp_liouville_ordem_s}, we define
$$
u(t'',x)=\sum_{k\in\mathbb{N}} \widehat{u}(t'',q_k) e^{iq_k x} =\sum_{k\in\mathbb{N}} \widehat{v}(t'',q_k)e^{ip_k\cdot t'} e^{iq_k x}.
$$
Thus $u\in D_s'(\mathbb{T}^{n+1})\setminus G^s(\mathbb{T}^{n+1})$ and
$$
\widehat{f}_j(t,q_k)\doteq \widehat{L_ju}(t,q_k)=
\begin{cases}
i(p_k^{(j)}+a_{j0}q_k) \widehat{v}(t'',q_k) e^{ip_k\cdot t'},&  j=1,\ldots, \ell\\
\widehat{g}_j(t'',q_k) e^{ip_k\cdot t'}, & j=\ell+1,\ldots,n.
\end{cases}
$$

In order to conclude the proof in this case, it is enough to show that for each $j$
$$
f_j(t,x)=\sum_{k\in\mathbb{N}} \widehat{f}_j(t'',q_k) e^{iq_k x}\in G^{s}(\mathbb{T}^{n+1}).
$$

Consider first $j=1,\ldots, \ell$. By taking $\epsilon>0$, as in Lemma \ref{lema_sequencia_esperta_exp_liouville_ordem_s}, there exist  $C>0$ and $h>0$ such that
$$
e^{-\frac{\epsilon}{4} q_k^{1/s}}|\partial_{t''}^{\alpha''}\widehat{v}(t'',q_k)|\leqslant C (\alpha''!)^s h^{|\alpha''|}
$$
and there exists $C_\epsilon>0$ such that
$$
|p_k|^{|\alpha'|} e^{-\frac{\epsilon}{4}q_k^{1/s}}\leqslant C_\epsilon^{|\alpha'|} (\alpha'!)^s,
$$
for all $\alpha=(\alpha',\alpha'')\in\mathbb{N}_0^{n}$, $t''\in \mathbb{T}^{n-\ell}$ and $k\in\mathbb{N}$. Then
\begin{eqnarray*}
	|\partial_t^{\alpha}\widehat{f}_j(t,q_k)|&=& |(p_k^{(j)}+a_{j0}q_k) \partial_{t''}^{\alpha''} \widehat{v}(t'',q_k) \partial_{t'}^{\alpha'} e^{ip_k\cdot t'}| \\[2mm]
	&\leqslant& qe^{-\epsilon q_k^{1/s}} C (\alpha''!)^s h^{|\alpha''|} e^{\frac{\epsilon}{4} q_k^{1/s}} C_\epsilon^{|\alpha'|}  (\alpha'!)^s e^{\frac{\epsilon}{4}q_k^{1/s}}\\[2mm]
	&\leqslant& q C h_1^{|\alpha|}(\alpha!)^s  e^{-\frac{\epsilon}{4}q_k^{1/s} },
\end{eqnarray*}
for all $\alpha=(\alpha',\alpha'')\in\mathbb{N}_0^{n},$ $t\in \mathbb{T}^n$ and $k\in\mathbb{N},$ where $h_1=\max\{1,C_\epsilon, h\}$.

If  $j=\ell+1,\ldots, n$, there exist positive constants $C$, $h$ and $\epsilon$ such that
$$
|\partial_{t''}^{\alpha''}\widehat{g}_j(t'',q_k)|\leqslant C (\alpha''!)^s h^{|\alpha''|}e^{-{\epsilon} q_k^{1/s}},
$$
for all $\alpha''\in\mathbb{N}_0^{\ell-n}$, $t''\in \mathbb{T}^{n-\ell}$ and $k\in \mathbb{N}$. Also,  there exists $C_\epsilon>0$ such that
$$|p_k|^{|\alpha'|} e^{-\frac{\epsilon}{2}q_k^{1/s}}\leqslant C_\epsilon^{|\alpha'|} (\alpha'!)^s$$ for all $\alpha'\in\mathbb{N}_0^{\ell}$, $t'\in \mathbb{T}^{\ell}$ and $k\in \mathbb{N}$.
Therefore, proceeding as before, we obtain a similar estimate to $|\partial_t^{\alpha}\widehat{f}_j(t,q_k)|$ for $j=\ell+1,\ldots, n$.

\section{Examples and final remarks}\label{section-final-remarks}

The study of the global hypoellipticity of the system \eqref{system-main-thm} is included in \cite{BCM93}. The abstract result obtained in \cite{BCM93}, together with Theorem 2.2 in \cite{Hou79}, imply an analogue result to Theorem \ref{main theorem} simply by replacing the terms {\it global $s-$hypoellipticity}  by {\it global hypoellipticity}, and {\it exponential Liouville vector of order $s>1$} by {\it Liouville vector}.

In particular, if the system \eqref{system-main-thm}, with $a_j$ and $b_j$ in $G^{s}(\mathbb{T}^n)$, is globally hypoelliptic then it is also globally $s-$hypoelliptic for any $s\geqslant 1$. In this section we present examples in which the reciprocal of this result  does not hold.

We start by recalling that the convergents $p_n/q_n$ of a continued fraction $\alpha = [a_1, a_2, a_3, \dots]$ are given by the following recurrence relations:
$p_1 = 1$, $q_1 = a_1$, $p_2 = a_2$, $q_2 = a_2a_1 +1$ and for $n \geqslant  3$,
$$p_n = a_np_{n-1} + p_{n-2}, \quad q_n = a_nq_{n-1} + q_{n-2}.
$$
\begin{proposition}\label{thm convergents}
	The convergents $p_n/q_n$ of a continued fraction $\alpha = [a_1, a_2, a_3, \dots]$ satisfy the following conditions:
	\begin{enumerate}
		\item[$i.$] $\dfrac{1}{(a_{n+1} + 2)q_{n}^{2}} < \Big|\alpha - \dfrac{p_n}{q_n} \Big| < \dfrac{1}{(a_{n+1})q_{n}^{2}},$ for all $n\in\mathbb{N},$\\[1mm]
		\item[$ii.$] If $p, q \in \mathbb{N}$ and $q \leqslant  q_n$, then $|p_n - \alpha q_n| < |p - q\alpha|.$
	\end{enumerate}
\end{proposition}

In \cite{Green72}, the author showed that the continued fraction $\alpha = [10^{1!}, 10^{2!}, 10^{3!}, \dots]$ is a Liouville number, but is not an exponential Liouville number. In the next result we improve the one just mentioned.

\begin{proposition}\label{not exp Liouville of order s}
	The number $\alpha =  [10^{1!}, 10^{2!}, 10^{3!}, \dots]$ is Liouville but is not an exponential Liouville number of order $s$, for any $s\geqslant 1$.
\end{proposition}

\begin{proof}
	The proof that $\alpha$ is a Liouville number can be found on page 162 in \cite{HarWri08}.
	
	Let  $p_n/q_n$ be the convergents of $\alpha = [10^{1!}, 10^{2!}, 10^{3!}, \dots]$. Let us prove that the following condition holds
	\begin{equation}\label{condicao_B}
	\forall \epsilon > 0,\; \exists N \in \mathbb{N}:\; n \geqslant  N \Rightarrow |p_n - \alpha q_n| \geqslant  e^{-\epsilon q_{n-1}^{{1}/{s} }}.
	\end{equation}
	
	Assume, for a moment, that \eqref{condicao_B} has been proved.  Then, it follows that
	\begin{equation}\label{condicao_A}
	\forall \epsilon > 0,\; \exists Q \in \mathbb{N}:\; p \in \mathbb{N}, q \geqslant  Q \Rightarrow |p - \alpha q| \geqslant  e^{-\epsilon q^{{1}/{s} }},
	\end{equation}
	which lets us thereby conclude that $\alpha$ is not an exponential Liouville number of order $s$.
	
	Indeed,  given $\epsilon>0$, by assuming \eqref{condicao_B} we consider $Q = q_{N - 1}$. Thus, if $p \in \mathbb{N}$ and $q \geqslant  Q$ we claim that $|p - \alpha q| \geqslant  e^{ -\epsilon q^{{1}/{s}}}$, otherwise
	$$
	|p - \alpha q| < e^{ -\epsilon q^{{1}/{s}} } \leqslant  e^{-\epsilon q_{N-1}^{ {1}/{s} }} \leqslant  |p_N - \alpha q_N|,
	$$
	which implies $q > q_{N}$ by  Proposition \ref{thm convergents}. Then
	$$
	|p - \alpha q| < e^{ -\epsilon q^{{1}/{s}} } \leqslant  e^{-\epsilon q_{N}^{ {1}/{s} }} \leqslant  |p_{N+1} - \alpha q_{N+1}|,
	$$
	resulting that $q > q_{N+1}$. By repeating this argument we obtain $q > q_{n}$ for all $n \geqslant  N$, which is a contradiction.
	
	We now proceed to prove condition \eqref{condicao_B}.
	By Proposition \ref{thm convergents} we have
	$$
	\dfrac{1}{(a_{n+1} + 2)q_n} \leqslant  |p_n - \alpha q_n|,\quad \forall n \in \mathbb{N},
	$$
	where $a_n=10^{n!}$.
	
	On the other hand,  since
	$
	a_{n+1} = 10^{(n+1)!} = (10^{n!})^{n+1} = a_n^{n+1},
	$
	we have $a_n^{n+2} = a_n^{n+1}a_n = a_{n+1} 10^{n!} \geqslant  (a_{n+1} + 2)$ for all $n$. Thus,
	$$
	\dfrac{1}{a_{n}^{n+2}q_n} \leqslant  \dfrac{1}{(a_{n+1} + 2)q_{n}} \leqslant  |p_n - \alpha q_n|,\quad \forall n \in \mathbb{N}.
	$$
	
	By the recurrence relation  $q_n = a_nq_{n-1} + q_{n-2}$, $n \geqslant  2$, we have $q_{n} \geqslant  a_{n}$ and thereby
	$$
	\dfrac{ 1 }{ q_{n}^{n+3} } \leqslant  |p_n - \alpha q_n|,\quad \forall n \geqslant  2.
	$$
	
	Therefore, given $\epsilon>0$, it is sufficient to prove that
	$$
	q_n^{n+3} \leqslant  e^{\epsilon q_{n-1}^{{ 1 }/{ s }}}, \quad n >> 0,
	$$
	where $n >> 0$ means that $n$ is large enough.
	
	Observe that for $n>>0$ we have
	$$
	2n\log(q_n) \leqslant  \epsilon q_{n-1}^{{ 1 }/{ s }} \Longrightarrow (n + 3)\log(q_n) \leqslant  \epsilon q_{n-1}^{{ 1 }/{ s }} \Longrightarrow q_n^{n+3} \leqslant  e^{\epsilon q_{n-1}^{{ 1 }/{ s }}}.
	$$
	
	Also, since  $q_n = a_nq_{n-1} + q_{n-2} \leqslant  (a_n + 1)q_{n-1} \leqslant  a_n^2 q_{n-1}$, $n \geqslant  3$, we have
	$$
	2n \log(q_n) \leqslant  4n \log(a_n) + 2n\log(q_{n-1}).
	$$
	Finally,  in order to complete the proof we show that
	
	\begin{enumerate}
		\item [(1)] $4n \log(a_n) \leqslant  \dfrac{\epsilon}{2} q_{n-1}^{{ 1 }/{ s }}$, for $n >> 0$ and
		\item[(2)] $2n\log(q_{n-1}) \leqslant  \dfrac{\epsilon}{2} q_{n-1}^{{ 1 }/{ s }}$,  for $n >> 0$.
	\end{enumerate}		
	
	\medskip
	
	To prove ($1$),  observe that $4n \log(a_n) = 4n \log(10^{n!}) = 4n (n!) \log(10)$ and that $q_{n-1}^{{ 1 }/{ s }} \geqslant  a_{n-1}^{{1}/{s}} = 10^{\frac{ (n-1)! }{ s }}$, 
	thus it is enough to show that
	$4n (n!) \log(10) \leqslant \dfrac{ \epsilon}{2} 10^{\frac{ (n-1)! }{ s }}$, for $n >> 0$, which follows from the fact that sequence
	$$
	n(n!)\left(10^{\frac{(n-1)!}{s}}\right)^{-1},
	$$
	goes to zero when $n$ goes to infinity.

	To prove $(2)$, note that  $x^{-1/s}\log(x)$ is a decreasing function for $x > e^s$. Then
	$$
	0\leqslant \dfrac{2n \log(q_{n-1})}{q_{n-1}^{{ 1 }/{ s }}} \leqslant  \dfrac{2n \log(10^{(n-1)!})}{10^{\frac{ (n-1)! }{ s }}}= 2\log(10)  \dfrac{n!}{10^{\frac{ (n-1)! }{ s }}} , \quad n >> 0.
	$$
	
	Since the sequence
	$$
	n!\left(10^{\frac{(n-1)!}{s}}\right)^{-1},
	$$
	goes to zero when $n$ goes to infinity, then the proof is complete.
	
\end{proof}

\begin{example}
	If $\alpha=[10^{1!}, 10^{2!}, 10^{3!}, \dots]$ and $\beta\in\mathbb{Q}$, then the system
	$$
	\begin{cases}
	L_1=\dfrac{\partial}{\partial t_1}+\alpha\dfrac{\partial}{\partial x}\\
	L_2=\dfrac{\partial}{\partial t_2}+\beta\dfrac{\partial}{\partial x},
	\end{cases} \
	$$
	is globally $s-$hypoelliptic, for all $s\geqslant1$, but is not globally hypoelliptic. Indeed, this follows immediately from Theorem
	\ref{main theorem}, Theorem 3.3 in  \cite{Berg99}  and Proposition \ref{not exp Liouville of order s}.
\end{example}

\begin{example}
	Consider the system
	$$
	\begin{cases}
	L_1=\dfrac{\partial}{\partial t_1}+\alpha\dfrac{\partial}{\partial x}\\
	L_2=\dfrac{\partial}{\partial t_2}+(\beta+i\sin(t_2))\dfrac{\partial}{\partial x},
	\end{cases}
	$$
	where $\alpha,\beta \in\mathbb{R}.$
	\begin{enumerate}
		\item[$\circ$] 	If $\alpha=[10^{1!}, 10^{2!}, 10^{3!}, \dots]$ then	this system is $s-$globally hypoelliptic for all $s\geqslant1$, but is not globally hypoelliptic.
		\item[$\circ$] In Lemma A in \cite{GPY93}, the authors show that there is an irrational number $\alpha$ that is not exponential Liouville, but it is exponential Liouville of order $s>1$. For this $\alpha$,  the system above is globally analytic hypoelliptic, but is not globally $s-$hypoelliptic, for any $s>1$.
		\item[$\circ$]  Also, in Lemma A in \cite{GPY93}, the authors show that it is possible to construct an irrational number $\alpha$ that is not exponential Liouville of order $s\geqslant1$, but it is exponential Liouville of order $s'>s$. For this $\alpha$,  the system above is globally $s-$ hypoelliptic, but is not globally $s'-$hypoelliptic.
	\end{enumerate}
	
\end{example}

\begin{remark}\label{prop_consequencia_prova_suficiencia}
	According to  Theorem \ref{main theorem}, if at least one of the vector fields
	$$
	L_j=\partial_{t_j}+(a_j+ib_j)(t_j)\partial_x, \ 1\leqslant j \leqslant n,
	$$
	is globally $s-$hypoel\-liptic on $\mathbb{T}_{t_j,x}^2$, then the system  \eqref{system-main-thm} is globally $s-$hypoel\-liptic on $\mathbb{T}^{n+1}$.
	The reciprocal, however,  does not hold.
	
	In fact, in Example $4.9$ in  \cite{Berg99}, there was presented  a couple of exponential Liouville numbers $\alpha$ and $\beta$ such that $(\alpha, \beta)$ is not an exponential Liouville vector. In the proof of this statement, the author observed that $(\alpha, \beta)$ is not even a Liouville vector.
	
	Therefore, these same numbers $\alpha$ and $\beta$ are exponential Liouville of order $s\geqslant  1$ and  $(\alpha, \beta)$ is not an exponential  Liouville vector of order $s\geqslant 1$. In particular, the system
	$$
	\begin{cases}
	L_1=\dfrac{\partial}{\partial t_1}+\alpha \dfrac{\partial}{\partial x}\\
	L_2=\dfrac{\partial}{\partial t_2}+\beta \dfrac{\partial}{\partial x}.
	\end{cases}
	$$
	is globally $s-$hypoelliptic, although the  vector fields $L_1$ and $L_2$ do not have this property.
\end{remark}

\begin{remark} When the coefficients of the system \eqref{system-main-thm} are not essentially real ($J=\emptyset$), it has already been shown that this system is globally $s-$hypoel\-liptic if and only if there is a function $t_j \in \mathbb{T}^1 \mapsto b_j(t_j)$ that does not change sign. In other words, the global $s-$hypoellipticity of \eqref{system-main-thm}  is equivalent to the validity of the Nirenberg-Treves condition $(P)$ for some vector field $L_j=\partial_{t_j}+(a_j+ib_j)(t_j)\partial_x$ on $\mathbb{T}_{t_j,x}^2$.
	
	This contrasts with the local theory. For example, the system
	$$
	\begin{cases}
	L_1= \dfrac{\partial}{\partial{t_1}} +it_1\dfrac{\partial}{\partial x}\\[4mm]
	L_2 = \dfrac{\partial}{\partial{t_2}} -it_2\dfrac{\partial}{\partial x}
	\end{cases}
	$$
	is locally $s-$hypoelliptic at the origin because the first integral $Z(t,x)=x+\frac{i}{2}(t_1^2-t_2^2)$ is an open map at the origin. But neither  $L_1$ nor $L_2$ is locally $s-$hypoelliptic at the origin.
\end{remark}

\begin{remark}
	Suppose again that the system \eqref{system-main-thm} has no real vector fields and consider the following closed $1-$form  on $\mathbb{T}^n$
	$$b(t)=\sum_{j=1}^nb_j(t_j)dt_j.$$
	
	

	Let $B$ be a global primitive of the pull-back $\Pi^*b$, where  $\Pi: \mathcal{T} \rightarrow \mathbb{T}^n$ is a minimal covering space of $\mathbb{T}^n$  with respect to the $1-$form $b$, in the sense of Definition 2.1 in \cite{BKNZ12}. It follows from Lemma 2.6 in \cite{BdMZ12} and Theorem \ref{main theorem} that:
	
	The system \eqref{system-main-thm}  is globally $s-$hypoelliptic if and only if $b$ is not exact and the sublevels $\Omega_r=\{t\in\mathcal{T}: B(t)<r\}$ and the superlevels $\Omega^r=\{t\in\mathcal{T}: B(t)>r\}$ are connected, for every $r\in \mathbb{R}$.
\end{remark}

\end{document}